\documentclass[12pt,a4paper,oneside]{article}
\usepackage[inner=3cm,outer=3 cm,top=2.5cm,bottom=2.5cm]{geometry}
\usepackage{amssymb}
\usepackage{amsthm} 
\usepackage{amsmath}
\usepackage{graphicx}
\usepackage{xcolor}
\usepackage{t1enc}
\usepackage[utf8]{inputenc}
\usepackage{enumerate}
\usepackage[UKenglish]{babel}
\usepackage{float}
\usepackage[unicode]{hyperref}
\usepackage{epstopdf}
\usepackage{indentfirst}
\usepackage{enumitem}
\setlist{leftmargin=2cm}
\usepackage{booktabs}
\usepackage{upgreek}
\usepackage{bm}
\usepackage{graphicx}
\usepackage{csquotes}
\usepackage{authblk}

\DeclareMathAlphabet{\mymathbb}{U}{bbold}{m}{n}

\hypersetup{
	colorlinks=false,
	linkcolor=red,
	urlcolor=blue,
	citecolor=gray
}						
\newtheorem{dfn}{Definition}[section]
\newtheorem{thm}[dfn]{Theorem}
\newtheorem{claim}[dfn]{Claim}

\newcommand\cH{{\mathcal{H}}}

\begin{document}

\title{On the chromatic number of pseudohemisphere hypergraphs  }
\author{Balázs István Szabó\thanks{Research supported by the EXCELLENCE-24 project no.~151504 Combinatorics and Geometry of the NRDI Fund.}}
\affil{ELTE E\"{o}tv\"{o}s Lor\'{a}nd University, Budapest}
\maketitle
 \begin{abstract}
  A pseudohemisphere hypergraph is a hypergraph $\mathcal{H}$ with an ordered set of vertices $V$ for which there exists an $ABA$-free hypergraph $\mathcal{F}$ on $V$ and a subset $X$ of $V$ such that the hyperedge set of $\mathcal{H}$ is a subset of $\{F\Delta X: F\in \mathcal{F}\cup \overline{\mathcal{F}}\}$. We prove that the chromatic number of pseudohemisphere hypergraphs is at most four.
 \end{abstract}

\section{Introduction}
The chromatic number of hypergraphs defined by halfplanes on a set of points is at most four \cite{3}. We get a natural generalization, if we replace halfplanes by pseudohalfpanes. The chromatic number of pseudohalfplane hypergraphs is also at most four \cite{1}.
Colorings and Helly type theorems for pseudohemisphere hypergraphs were investigated in \cite{1},\cite{2}. It was shown that the vertices of a pseudohemisphere hypergraph can be colored with $k$ colors such that every hyperedge containing at least $4k-3$ vertices contains vertices of all $k$ colors. We show that the vertices of a pseudohemisphere hypergraph can be colored with four colors such that every hyperedge containing at least two vertices contains two different colors. Note that, using the fact that the chromatic number of pseudohalfplane hypergraphs is at most four we can easily obtain eight as an upper bound for the chromatic number of pseudohemisphere hypergraphs, since we can color the vertices of $X$ and $V\setminus X$ using four colors each. We improve this upper bound to four, which is tight. The lower bound is trivial, as $K_{4}$ can be realized as a pseudohalfplane hypergraph and every pseudohalfplane hypergraph is a pseudohemisphere hypergraph.
We begin by recalling some definitions and previous results that will be needed later \cite{1}.
\begin{dfn}
	A hypergraph $\mathcal{F}$ with an ordered set of vertices is called $ABA$-free if it does not contain two hyperedges $A$ and $B$ for which there exist three vertices $v_{1}<v_{2}<v_{3}$ such that $v_{1},v_{3}\in A\setminus B$ and $v_{2}\in B\setminus A$. A vertex $v$ is called unskippable if there exists a hyperedge $F$ such that $v\notin F$ and $v$ is between the vertex of $F$ with the smallest index and the vertex of $F$ with largest index.
\end{dfn}

\begin{claim}\label{unskippable_vertex}\cite{1}
	If $\mathcal{F}$ is an $ABA$-free hypergraph, then every hyperedge of $\mathcal{F}$ contains an unskippable vertex. 
	
\end{claim}

If $\mathcal{H}$ is a hypergraph defined by the vertex set $V$ and the hyperedge set $\mathcal{E}$, then we denote the complement hypergraph by $\overline{\mathcal{H}}$, its vertex set is the vertex set of $\mathcal{H}$ and the hyperedge set is $\{V\setminus E: E \in \mathcal{E}\}$. For a given subset of vertices $X\subseteq V$, the subhypergraph induced by $X$ is the hypergraph $\mathcal{H}[X]=(X,\mathcal{E}_{X})$, where $\mathcal{E}_{X}=\{E\cap X: E\in \mathcal{E}\}$.

\begin{dfn}
A hypergraph $\mathcal{H}$ with an ordered set of vertices $V$ is called a pseudohalfplane hypergraph if there exists an $ABA$-free hypergraph $\mathcal{F}$ such that the set of the vertices of $\mathcal{H}$ is the vertex set of $\mathcal{F}$ and the set of hyperedges of $\mathcal{H}$ is a subset of $\mathcal{F}\cup\overline{\mathcal{F}}$. The hyperedges of $\mathcal{H}$ from $\mathcal{F}$ (resp. $\overline{\mathcal{F}}$) are called topsets (resp. bottomsets) and the unskippable vertices of $\mathcal{F}$ (resp. $\overline{\mathcal{F}}$) are called topvertices (resp. bottomvertices). 
Let $T(\mathcal{H})$ denote the set of the topvertices of $\mathcal{H}$ and $B(\mathcal{H})$ the set of the bottomvertices of $\mathcal{H}$.
The vertices that are topvertices or bottomvertices are called extremal vertices and the vertices that are neither topvertices nor bottomvertices are called non-extremal vertices.
\end{dfn}

\begin{thm}\label{pseudohalfplane}\cite{1}
	Let $\mathcal{H}$ be a pseudohalfplane hypergraph. Then a topset (resp. bottomset) intersects the set of the topvertices (resp. bottomvertices) in an interval. Furthermore, if $H\in E(\mathcal{H})$ is a topset (resp. bottemset) and $v\in H$ such that $v\notin T(\mathcal{H})$ (resp. $v\notin B(\mathcal{H})$), then $H$ must contain at least one of the two topvertices before and after $v$ that are closest to $v$. 
\end{thm}

\begin{dfn}
A pseudohemisphere hypergraph is a hypergraph $\mathcal{H}$ with an ordered set of vertices $V$ for which there exists an $ABA$-free hypergraph $\mathcal{F}$ on $V$ and a subset $X$ of $V$ such that the hyperedge set of $\mathcal{H}$ is a subset of $\{F\Delta X: F\in \mathcal{F}\cup \overline{\mathcal{F}}\}$.
\end{dfn}

In this paper we will use the abstract definition of pseudohemisphere hypergraphs, but there are nice geometric representations \cite{1}.

\begin{thm}\label{pseudohemisphere}
	Let $\mathcal{H}$ be a pseudohemisphere hypergraph. Then $\mathcal{H}$ admits a proper 4-coloring.
\end{thm}

\section{Proof of Theorem \ref{pseudohemisphere}}
\begin{proof}	

	We can assume that $\mathcal{H}=\{F\Delta X :F\in \mathcal{F}\cup \overline{\mathcal{F}} \}$. Let $Y$ denote the vertex set $V\setminus X$.
    Take the induced subhypergraphs $\mathcal{H}_{1}=(\mathcal{F}\cup\mathcal{\overline{F}})[Y]$ and $\mathcal{H}_{2}=(\mathcal{F}\cup\overline{\mathcal{F}})[X]$.
	Then $\mathcal{H}_{1}$ and $\mathcal{H}_{2}$ are pseudohalfplane hypergraphs. Each hyperedge $H\in \mathcal{H}$ can be written the following way: $H=(F_H\cap Y)\cup (X\setminus (F_H\cap X))$ for some $F_H \in \mathcal{F}\cup\mathcal{\overline{F}}$, therefore $F_{H}=(H\cap Y)\cup (X\setminus (H\cap X))$. If $F_{H}$ is a topset (resp. bottomset), then $H\cap Y$ is a topset (resp. bottomset) in $\mathcal{H}_{1}$ and $H\cap X$ is a bottomset (resp. topset) in $\mathcal{H}_{2}$. We can partition the set of hyperedges of $\mathcal{H}$ with at least two vertices into two parts according to whether they intersect $Y$ or $X$ in at least two vertices or both of them in exactly one vertex. If we give a proper 4-coloring for the hypergraphs $\mathcal{H}_{1}$ and $\mathcal{H}_{2}$, then every hyperedge which intersects $Y$ or $X$ in at least two vertices is colored properly. To achieve this, thanks to Theorem $\ref{pseudohalfplane}$, for each of $\cH_1$ and $\cH_2$ it is sufficient to give a $4$-coloring in which there is no two consecutive topvertices nor two consecutive bottomvertices that have the same color and every non-extremal vertex has a different color from the two topvertices and the two bottomvertices before and after it that are closest to it. We will give a proper 4-coloring of the hypergraphs $\mathcal{H}_{1}$ and $\mathcal{H}_{2}$ in which the extremal vertices of $\mathcal{H}_{1}$ and $\mathcal{H}_{2}$ are colored with three colors such that no two consecutive topvertices nor two consecutive bottomvertices have the same color and we color the non-extremal vertices with the fourth color. This guarantees that the hyperedges in $\mathcal{H}_{1}$ and $\mathcal{H}_{2}$ with at least two vertices are colored properly. 

    \begin{claim}\label{top-bottom}
    If a hyperedge $H$ intersects $Y$ and $X$ in exactly one vertex and $F_{H}$ is a topset (resp. bottomset), then $H\cap Y$ is a topvertex (resp. bottomvertex)  of $\mathcal{H}_{1}$ and $H\cap X$ is a bottomvertex (resp. topvertex) of $\mathcal{H}_{2}$.
    \end{claim}
    
	Let $a$ be the first vertex of $\mathcal{H}_{1}$ and $b$ be the last vertex of $\mathcal{H}_{1}$. Similarly let $c$ be the first vertex of $\mathcal{H}_{2}$ and $d$ be the last vertex of $\mathcal{H}_{2}$. If $F$ is a hyperedge of $\mathcal{F}$ (resp. $\overline{\mathcal{F}}$) such that $F\cap Y$ is a single vertex $v$ which is distinct from $a$ and $b$, then $v$ is a topvertex (resp. bottomvertex) of $\mathcal{H}_{1}$ and $\overline{F}\cap Y=Y\setminus \{v\}$. As $a,b\in Y\setminus \{v\}$, Theorem \ref{pseudohalfplane} implies that $v$ cannot be a bottomvertex (resp. topvertex) of $\mathcal{H}_{1}$. Similarly, if $F$ is a hyperedge of $\mathcal{F}$ (resp. $\overline{\mathcal{F}}$) such that $F\cap X$ is a single vertex $v$ which is distinct from $c$ and $d$, then $v$ is a bottomvertex (resp. topvertex) of $\mathcal{H}_{2}$ and $\overline{F}\cap Y=Y\setminus \{v\}$. Again, as $a,b\in Y\setminus \{v\}$, Theorem \ref{pseudohalfplane} implies that $v$ cannot be a topvertex (resp. bottomvertex) of $\mathcal{H}_{2}$.

	We can assume that apart from $a$ and $b$, there is no vertex of $\mathcal{H}_{1}$ which is both a topvertex and a bottomvertex of $\mathcal{H}_{1}$. If there would be such a vertex $v$ then we could add $\{v\}$ as a topset and $V(\mathcal{H}_{1})\setminus \{v\}$ as a bottomset to the hypergraph $\mathcal{H}_{1}$, in the resulting hypergraph $v$ is only a topvertex and if we give a proper 4-coloring to this hypergraph, then it is a proper 4-coloring of $\mathcal{H}_{1}$ as well. 
	Similarly, we can assume that apart from $c$ and $d$, there is no vertex of $\mathcal{H}_{2}$ which is both a topvertex and a bottomvertex of $\mathcal{H}_{2}$.

The hyperedges intersecting $Y$ and $X$ in exactly one vertex have the following useful property.
\begin{claim}\label{non-crossing}
	Let $H_{1}=\{u_{1},v_{1}\}$ and $H_{2}=\{u_{2},v_{2}\}$ two hyperedges of $\mathcal{H}$ such that $u_{1},v_{1},u_{2},v_{2}$ are distinct vertices, $u_{1},u_{2}\in Y$, $v_{1},v_{2}\in X$ and $u_{1}$ has smaller  index than $u_{2}$. Suppose that either $F_{H_{1}},F_{H_{2}}\in \mathcal{F}$ or $F_{H_{1}},F_{H_{2}}\in \mathcal{\overline{F}}$. Then $v_{1}$ has larger index than $v_{2}$.
\end{claim}
\begin{proof}
Assume for a contradiction that $u_{1}$ has smaller index than $u_{2}$ and $v_{1}$ has smaller index than $v_{2}$. If $u_{1}$ has smaller index than $v_{1}$, then $u_{1}<v_{1}<v_{2}$, while $u_{1},v_{2}\in F_{H_{1}}\setminus F_{H_{2}}$ and $v_{1}\in F_{H_{2}}\setminus F_{H_{1}}$, so $\mathcal{F}$ or $\mathcal{\overline{F}}$ is not $ABA$-free, this is a contradiction. Similarly, if $u_{1}$ has larger index than $v_{1}$, then $v_{1}<u_{1}<u_{2}$ and $v_{1},u_{2}\in F_{H_{2}}\setminus F_{H_{1}}$, while $u_{1}\in F_{H_{1}}\setminus F_{H_{2}}$, so $\mathcal{F}$ or $\mathcal{\overline{F}}$ is not $ABA$-free, this is also a contradiction. This implies that $v_{1}$ has larger index than $v_{2}$.
\end{proof}
If $H_{1}$ and $H_{2}$ satisfy the conditions of Claim \ref{non-crossing} then we say that $H_{1}$ and $H_{2}$ satisfy the \emph{non-crossing property}.

We give a graph $G$ and a mapping $f:V(\mathcal{H})\rightarrow V(G)$ such that for every hyperedge of $\mathcal{H}$ its image contains the vertices of at least one edge of the graph $G$. It implies that if $G$ admits a proper $4$-coloring, then $\mathcal{H}$ is also proper $4$-colorable (given by $f^{-1}$). We take the extremal vertices of $\mathcal{H}_{1}$ and $\mathcal{H}_{2}$ and the vertex $n_{i}$ will represent the non-extremal vertices of $\mathcal{H}_{i}$ where $i=1,2$. 
The edge set of $G$ contains the following edges
(recall that $a$ is the first vertex and $b$ is the last vertex of $\mathcal{H}_{1}$ and $c$ is the first vertex and $d$ is the last vertex of $\mathcal{H}_{2}$): 
\begin{itemize}
	\item $uv\in E(G)$ if $u$ and $v$ are consecutive topvertices of $\mathcal{H}_{i}$ where $i=1,2$
	\item $uv\in E(G)$ if $u$ and $v$ are consecutive bottomvertices of $\mathcal{H}_{i}$ where $i=1,2$
	\item $n_{1}v\in E(G)$ if $v$ is an extremal vertex of $\mathcal{H}_{1}$
	\item $n_{2}v\in E(G)$ if $v$ is an extremal vertex of $\mathcal{H}_{2}$
	\item $ad,bc\in E(G)$
	\item $uv\in E(G)$ if $u$ is a vertex of $\mathcal{H}_{1}$, $v$ is a vertex of $\mathcal{H}_{2}$ and $\{u,v\}$ is a hyperedge of $\mathcal{H}$ (By Claim \ref{top-bottom}, we know that $u$ is a topvertex of $\mathcal{H}_{1}$ and $v$ is a bottomvertex of $\mathcal{H}_{2}$ or $u$ is a bottomvertex of $\mathcal{H}_{1}$ and $v$ is a topvertex of $\mathcal{H}_{2}$.)
\end{itemize}

\begin{figure}[h!]
	\centering
	\includegraphics[width=0.8\textwidth]{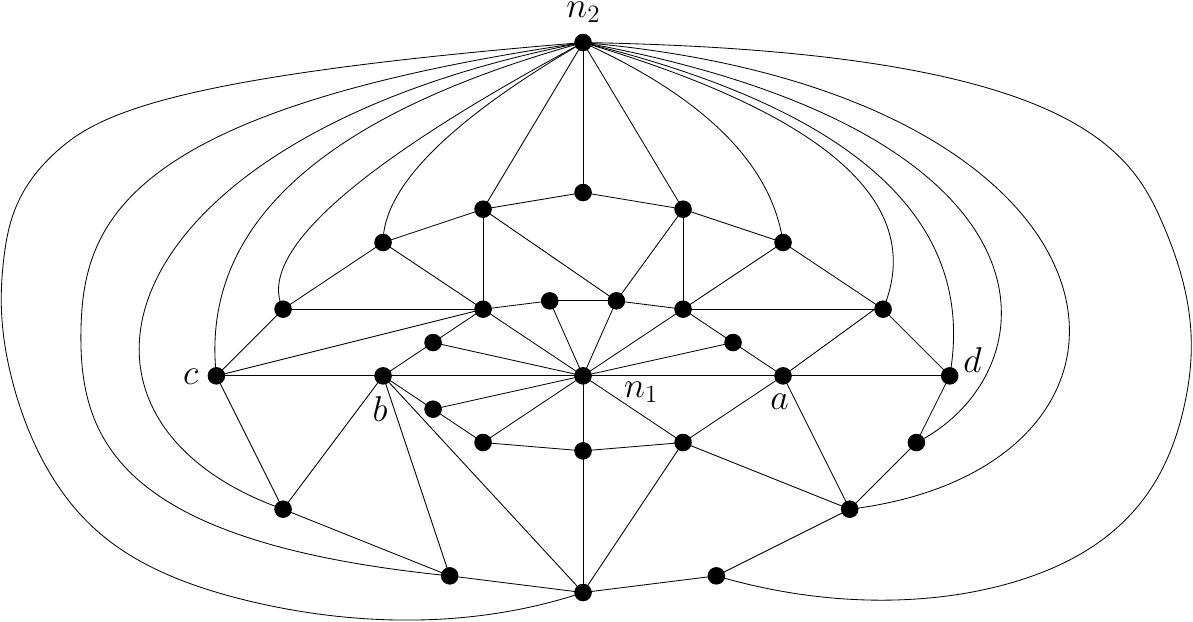}
	\caption{Planar drawing of $G$}
	\label{G plane}
	
\end{figure}

Figure \ref{G plane} shows a planar drawing of $G$. The upper $a-b$ path contains the topvertices of $\mathcal{H}_{1}$ and the lower one contains the bottomvertices of $\mathcal{H}_{1}$. The upper $c-d$ path contains the bottomvertices of $\mathcal{H}_{2}$ and the lower one contains the topvertices of $\mathcal{H}_{2}$. The absence of edge crossings follows from Claim \ref{non-crossing}. It follows from the Four Color Theorem that $G$ is proper $4$-colorable. \newline
Let $H$ be a hyperedge of $\mathcal{H}$. If $H$ intersects $Y$ (resp. $X$) in at least two vertices, then it follows from Claim \ref{unskippable_vertex} that $H$ contains an extremal vertex. If $H$ contains a second extremal vertex of $\mathcal{H}_{1}$ (resp. $\mathcal{H}_{2}$), then it follows from Theorem \ref{pseudohalfplane} that $H$ contains two consecutive topvertices or two consecutive bottomvertices of $\mathcal{H}_{1}$ (resp. $\mathcal{H}_{2}$), so $H$ contains an edge of $G$. If $H$ does not contain a second extremal vertex of $\mathcal{H}_{1}$ (resp. $\mathcal{H}_{2}$), then $H$ contains a non-extremal vertex $v$ and it follows from Theorem \ref{pseudohalfplane} that $H$ must contain at least one of the extremal vertices before and after $v$ that are closest to $v$, so $H$ contains an edge of $G$.
If $H$ intersects $Y$ and $X$ in exactly one vertex, then it also contains an edge of $G$, which connects a vertex on the upper $a-b$ path to a vertex on the upper $c-d$ path or a vertex on the lower $a-b$ path to a vertex on the lower $c-d$ path. Thus, every hyperedge of $\mathcal{H}$ contains an edge of $G$. Since $G$ admits a proper $4$-coloring, we proved that $\mathcal{H}$ is proper $4$-colorable.

\end{proof}

\section{Acknowledgment}
I am grateful to Balázs Keszegh for the valuable discussions, helpful ideas and for reading this manuscript.

\end{document}